\documentclass[a4paper,11pt]{article}

\usepackage{amsmath,amsthm,amssymb}
\usepackage[top=3cm,bottom=3cm,left=3cm,right=3cm]{geometry} 
\usepackage[T1]{fontenc} 
\usepackage{url}
\usepackage{hyperref}

\DeclareMathOperator{\im}{im}

\DeclareMathOperator{\End}{End}

\DeclareMathOperator{\Hom}{Hom}

\DeclareMathOperator{\tr}{Tr}

\DeclareMathOperator{\Rm}{Rm}
\DeclareMathOperator{\Ric}{Ric}
\DeclareMathOperator{\vol}{vol}

\DeclareMathOperator{\Sym}{Sym}

\DeclareMathOperator{\Id}{Id}

\DeclareMathOperator{\Diff}{Diff}

\DeclareMathOperator{\SO}{SO}

\DeclareMathOperator{\Lie}{Lie}

\usepackage{mathrsfs}

\newcommand{\R}{\mathbb R}
\newcommand{\C}{\mathbb C}

\newcommand{\G}{\mathscr G}

\newcommand{\A}{\mathscr A}

\newcommand{\D}{\mathscr D}
\newcommand{\M}{\mathscr M}

\newcommand{\SEH}{S_{\mathrm{EH}}}

\newcommand{\diff}{\mathrm{d}}
\newcommand{\del}{\partial}

\newcommand{\dvol}{\mathrm{dvol}}

\newcommand{\so}{\mathfrak{so}}

\renewcommand{\P}{\mathbb P}

\renewcommand{\geq}{\geqslant}

\theoremstyle{plain}
	\newtheorem{theorem}{Theorem}
	\newtheorem{proposition}[theorem]{Proposition}
	\newtheorem{lemma}[theorem]{Lemma}

\theoremstyle{definition}
	\newtheorem{definition}[theorem]{Definition}

\theoremstyle{plain}
	\newtheorem*{theorem*}{Theorem}
	\newtheorem*{proposition*}{Proposition}
	\newtheorem*{lemma*}{Lemma}
	\newtheorem*{corollary*}{Corollary}
	\newtheorem*{conjecture*}{Conjecture}
\theoremstyle{definition}
	\newtheorem*{definition*}{Definition}
	\newtheorem*{remark*}{Remark}
	\newtheorem*{remarks*}{Remarks}

\usepackage{authblk}

\numberwithin{equation}{section}
\numberwithin{theorem}{section}

\begin{document}

\title{Local rigidity of Einstein 4-manifolds satisfying a chiral curvature condition}

\author[1]{Joel Fine \thanks{JF was supported by ERC consolidator grant 646649 ``SymplecticEinstein'' and EoS grant 30950721.}}
\author[2]{Kirill Krasnov}
\author[3]{Michael Singer}
\affil[1]{Universit\'e libre de Bruxelles, Belgium}
\affil[2]{University of Nottingham, UK}
\affil[3]{University College London, UK}

\date{\today}

\maketitle

\abstract{Let $(M,g)$ be a compact oriented Einstein 4-manifold. Write $R_+$ for the part of the curvature operator of $g$ which acts on self-dual 2-forms. We prove that if $R_+$ is negative definite then $g$ is locally rigid: any other Einstein metric near to $g$ is isometric to it. This is a chiral generalisation of Koiso's Theorem, which proves local rigidity of Einstein metrics with negative sectional curvature. Our hypotheses are roughly one half of Koiso's. Our proof uses a new variational description of Einstein 4-manifolds, as critical points of the so-called pure connection action $S$. The key step in the proof is that when $R_+<0$, the Hessian of $S$ is strictly positive modulo gauge.}

\tableofcontents
\vfill

\section{Introduction}

\subsection{Statement of the main result}

A Riemannian metric $g$ is called \emph{Einstein} if its Ricci curvature is constant:
\begin{equation}
\Ric(g) = \Lambda g, \quad \Lambda \in \R
\label{Einstein}
\end{equation}
If $g$ is Einstein and $\phi$ is a diffeomorphism, then $\phi^*g$ is also Einstein. When considered modulo diffeomorphisms however equation~\eqref{Einstein} is elliptic of index zero. This means that naively one might expect Einstein metrics to be isolated modulo diffeomorphisms, or ``locally rigid''. It is important to note that there are many situations in which Einstein metrics are not rigid: hyperbolic metrics on surfaces, many K\"ahler--Einstein metrics, holonomy $G_2$ metrics, \ldots. The aim of this article is to give a curvature condition in dimension 4 which ensures this naive expectation of local rigidity is true.

We begin with the definition.
\begin{definition}~
\begin{enumerate}
\item
Given an Einstein metric $g$, let $D_g$ denote the linearisation of the map $g \mapsto \Ric(g) - \Lambda g$ at $g$. An \emph{infinitesimal Einstein deformation of $g$} is a section $h$ of $S^2T^*M$ for which $D_g(h)=0$. 
\item
An Einstein metric $g$ is called \emph{locally rigid} if for any infinitesimal Einstein deformation $h$ of $g$ there is a vector field $v$ such that $h =L_vg$.
\end{enumerate}
\end{definition}

Let $(M,g)$ be a locally rigid compact Einstein metric. It can be shown that if $\widehat{g}$ is another Einstein metric sufficiently close to $g$ (in an appropriate H\"older space, say) then there is a diffeomorphism $\phi$, close to the identity, with $\hat{g} = \phi^*g$. (For a proof see, for example, \cite{Olivier}.) This justifies the use of the term ``local rigidity'' rather than just ``infinitesimal rigidity'' in the preceding definition. 

To state our local rigidity result, we first recall the decomposition of the curvature tensor of an oriented Riemannian 4-manifold $(M^4,g)$. In this dimension the 2-forms split $\Lambda^2 = \Lambda^+ \oplus \Lambda^-$ as the $\pm1$-eigenspaces of the Hodge star. This induces a decomposition of the curvature operator:
\begin{equation}\label{curvature-decomposition}
\Rm 
=
\begin{pmatrix}
R_{+} & C^* \\
C & R_{-}
\end{pmatrix}
\end{equation}
Here $R_{+} \in \End(\Lambda^+)$ and $R_{-} \in \End(\Lambda^-)$ are self-adjoint, whilst $C \in \Hom(\Lambda^+,\Lambda^-)$. This relates to the standard decomposition into the scalar, trace-free Ricci and Weyl curvatures (denoted $R, \Ric^0$ and $W$ respectively) as follows: $R = 4\tr(R_{+}) = 4 \tr(R_{-})$, $\Ric^0 = C$, and $W = R_{+}^0 \oplus R_{-}^0$ (where $E^0$ denotes the trace-free part of $E$ and where we have used a suitable identification of $\Hom(\Lambda^+,\Lambda^-)$ with symmetric trace-free bilinear forms on $TM$). In particular, when $g$ is Einstein, $\Rm = R_{+} \oplus R_{-}$.  

We now state our main result. In the statement the expression $R_+<0$ means that the quadratic form $(R_+(w),w)$ is negative definite.

\begin{theorem}\label{main-result}
Let $g$ be an Einstein metric on a compact orientable 4-manifold, with either $R_{+}<0$ or $R_{-}<0$. Then $g$ is locally rigid. 
\end{theorem}

To put our main result in context, we first describe one situation in which local rigidity has already been established, namely Koiso's Theorem. 

\begin{theorem}[Koiso \cite{Koiso}]\label{Koiso}
Let $g$ be an Einstein metric on a compact orientable manifold of dimension at least 3,  with negative curvature (i.e.\ all sectional curvatures of $g$ are negative). Then $g$ is locally rigid.
\end{theorem}

Theorem~\ref{main-result} is a chiral version of Koiso's Theorem, in dimension 4, with roughly ``one half'' the hypotheses. To see this, we relate the sectional curvatures to $R_{+}$ and $R_{-}$. Let $u,v$ be orthogonal unit tangent vectors, write $u^\flat,v^\flat$ for their metric dual covectors and $\alpha = (u^\flat \wedge v^\flat)^+$ and $\beta = (u^\flat\wedge v^\flat)^-$. When $g$ is Einstein the sectional curvature in the plane spanned by $u$ and $v$ is
\[
\sec(u,v)
=
\left\langle R_{+}(\alpha), \alpha  \right\rangle
+
\left\langle R_{-}(\beta) , \beta \right\rangle
\]
By letting $u,v$ vary, we can make $\alpha$ and $\beta$ take any values in the spheres in $\Lambda^+$ and $\Lambda^-$. Write $\lambda$ for the largest eigenvalue of $R_{+}$ and $\mu$ for the largest eigenvalue of $R_-$. The requirement that an Einstein metric have negative sectional curvatures is then equivalent to $\lambda + \mu <0$, whilst the hypothesis in our main result, Theorem~\ref{main-result}, is that either $\lambda<0$ or $\mu<0$. 

We remark that \emph{global} rigidity results are much more difficult to establish. The only known results are the following: if $(M,g)$ is a compact hyperbolic or complex-hyperbolic 4-manifold then $g$ is the unique Einstein metric on $M$ up to scale and diffeomorphism. In the hyperbolic case this is due to Besson--Courtois--Gallot~\cite{Besson-Courtois-Gallot}, whilst the complex-hyperbolic case is due to LeBrun~\cite{LeBrun}. 

\subsection{Outline of the proof}

Consider first Koiso's Theorem. To prove local rigidity of a negatively curved Einstein metric, Koiso uses a clever Bochner trick. Given an infinitesimal Einstein variation $h$, it is possible to find a vector field $v$ so that $h-L_vg$ solves an equation of the form $\mathscr{D}(h-L_vg) = 0$ where $\mathscr{D}$  is a Laplace type operator. Koiso finds a Weitzenb\"ock formula $\mathscr{D} = d^*d + B$, where $d$ is an appropriate first order operator. Crucially the remainder term $B$ is positive when the sectional curvatures of $g$ are negative. It follows that $h = L_vg$. 

Our proof of local rigidity is different. We use a special variational formulation of 4-dimensional Einstein metrics, due to the second author \cite{Krasnov} and independently, albeit in a weaker form, to the first author \cite{Fine}. A full description of this ``pure connection formalism'' which is accessible to mathematicians appears in \cite{Fine-Panov-Krasnov}. We give a brief summary here and more details in \S\ref{plebanski-to-pure-connection}, choosing a description which is better suited to our purposes here than that given in the above references. 

Let $E \to M^4$ be an $\SO(3)$-vector bundle over $M$. The following definition is crucial to what follows. It first appeared in this form in \cite{Fine-Panov}, but is actually a special case of Weinstein's ``fat connections'' \cite{Weinstein}.

\begin{definition}\label{definite}
An $\SO(3)$-connection $A$ on $E$ is called \emph{definite} if the curvature $F_A(u,v)$ is non-zero whenever $u$ and $v$ are linearly independent tangent vectors on~$M$. 
\end{definition}

Examples can be found from Riemannian geometry. Let $(M,g)$ be an oriented Riemannian 4-manifold, take $E= \Lambda^+$ and $A$ to be the Levi-Civita connection of $g$. Then $A$ is definite precisely when $R_{+}^2 - C^*C$ is a definite endomorphism of $\Lambda^+$ (we use here the notation of the curvature decomposition~\eqref{curvature-decomposition}; the proof that these metrics give definite connections is in \cite{Fine-Panov}). Notice that Einstein metrics have $C=0$ and so when $R_+$ is also non-degenerate (all eigenvalues are non-zero) these metrics give definite connections. Examples include the constant curvature metrics on hyperbolic space or the four-sphere or, more generally, anti-self-dual Einstein metrics with non-zero scalar curvature. 

When $A$ is definite it determines in a canonical way a Riemannian metric $g_A$. The metric tensor $g_A$ is built pointwise out of the curvature $F_A$. The precise formula for $g_A$ is given below in \S\ref{plebanski-to-pure-connection} but it is not important for this introductory discussion. We note however that if $A$ is the Levi-Civita connection on $\Lambda^+$ for a Riemannian metric $g$ as in the previous paragraph, then it is \emph{not} in general true that $g_A=g$. This happens precisely when $g$ is Einstein and $R_+$ is definite, with all eigenvalues nonzero and having the same sign.
 
Write $S(A) = \frac{\Lambda}{2} \int_M \dvol(g_A)$ where $\dvol(g_A)$ is the volume form of the metric $g_A$. The key fact we need is:

\begin{theorem}[\cite{Krasnov}, see also \cite{Fine-Panov-Krasnov}] \label{action-principle}
If $A$ is a critical point for $S$ then $g_A$ is an Einstein metric, with $R_+$ a definite endomorphism of $\Lambda^+$. Conversely, all such Einstein metrics arise this way, with $A$ corresponding to the Levi-Civita connection on~$\Lambda^+$. 
\end{theorem}

We now return to the proof of Theorem~\ref{main-result}. Let $g$ be  an Einstein metric for which $R_+<0$. (Note that reversing the orientation of $M$ swaps $R_+$ and $R_-$ and so when proving Theorem~\ref{main-result} it suffices to consider $R_+<0$.) The Levi-Civita connection of $g$ on $E=\Lambda^+$ gives a critical point of $S$. The key to proving Theorem~\ref{main-result} is to show that $R_+<0$ ensures that the Hessian of $S$ at $A$ is \emph{strictly positive} modulo gauge. So $A$ is an isolated local minimum of $S$ (again, modulo gauge). We then show that this implies the original Einstein metric $g$ is locally rigid. 

The article is laid out as follows. In \S\ref{plebanski-to-pure-connection} we describe the pure connection formalism, showing how to derive it from the Plebanski formulation of Einstein's equations. We then use this derivation in \S\ref{hessian} to compute the Hessian of the pure connection action. \S\ref{infinitesimal-rigidity} uses this to prove local rigidity of $g$.

\subsection{Remark}

We hope this proof highlights an important difference between the pure connection action $S$ and the more classical variational approach to Einstein metrics, given by the Einstein--Hilbert action $\SEH$. At a critical point, the Hessian of $\SEH$ is always indefinite, having an \emph{infinite} number of both positive and negative eigenvalues. This makes $\SEH$ extremely difficult to analyse from the point of view of the calculus of variations. The pure connection action does not suffer from this problem. Irrespective of curvature assumptions, at a critical point its Hessian is elliptic (modulo gauge), with at most \emph{finitely} many negative eigenvalues \cite{Fine-Panov-Krasnov}. It remains to be seen if techniques from the calculus of variations can be used to greater effect with the pure connection action than has been possible for the Einstein--Hilbert action.

\section{From Plebanski to the pure connection action}
\label{plebanski-to-pure-connection}

In \cite{Plebanski} Plebanski gave an alternative formulation of Einstein 4-manifolds, as critical points of a certain functional, now called the Plebanski action. By ``integrating out'' some of the variables of the Plebanski action, one arrives at the pure connection action of \cite{Krasnov}. This point of view is central to our computation of the Hessian of the pure connection action in \S\ref{hessian}.

Before describing Plebanski's action, we first recall some basic facts about 4-dimensional Riemannian geometry, with an emphasis which, whilst it may seem unusual at first sight, is well suited to our later purposes.

\subsection{Some foundational 4-dimensional facts}

Let $M$ be an oriented 4-manifold. The wedge product on 2-forms $\Lambda^2 \otimes \Lambda^2 \to \Lambda^4$  is a non-degenerate symmetric bilinear form on $\Lambda^2$ with values in $\Lambda^4$. We interpret this as a fibrewise conformal structure on $\Lambda^2$ with signature $(3,3)$. Given any Riemannian metric $g$ on $M$, this inner product makes self-dual and anti-self-dual forms orthogonal, is positive definite on $\Lambda^+_g$ and negative definite on $\Lambda^-_g$. Conversely, given any 3-dimensional sub-bundle $P \subset \Lambda^2$ on which the wedge-product is positive definite, there is a unique conformal class of metrics $[g]$ on $M$ for which $P =\Lambda^+_g$. A Riemannian metric is thus equivalent to the data of a positive definite sub-bundle $P$ together with a choice of volume form. 

The abstract isomorphism class of the vector bundle $\Lambda^+_g$ is independent of $g$; they are all isomorphic to some fixed reference $\SO(3)$-bundle $E \to M$. We fix this choice of $E$ once and for all. For any Riemannian metric $g$ on $M$, there is a (non-unique) vector bundle homomorphism $\Sigma \colon E \to \Lambda^2$ which is an isometry onto $\Lambda^+_g$. In fact it is convenient to fix the scale so that unit-length vectors in $E$ are sent to self-dual 2-forms of length $\sqrt{2}$. The metric $g$ can be recovered from $\Sigma$ by setting $\Lambda^+(g) = \Sigma(E)$ and taking the volume form $\mu = \frac{1}{2}\Sigma(e) \wedge \Sigma(e)$ where $e$ is any unit-length element of $E$. This leads us to the following definition.

\begin{definition}
$\Sigma \in \Omega^2(M,E^*)$ is called \emph{wedge-orthogonal} if when thought of as a vector bundle homomorphism $\Sigma \colon E \to \Lambda^2$ there is a positive 4-form $\mu$ such that for any pair $e_1, e_2 \in E$, we have $\Sigma(e_1)\wedge \Sigma(e_2) = 2 \left\langle  e_1,e_2 \right\rangle \mu$.

Given a wedge-orthogonal $\Sigma$, there is a unique Riemannian metric $g_\Sigma$ on $M$ for which $\Sigma$ takes values in $\Lambda^+_{g_\Sigma}$ and with $\dvol(g_\Sigma)=\mu$. 
\end{definition}

In the Plebanski formulation, we will use wedge-orthogonal 2-forms to parametrise metrics. Next we need to describe the Levi-Civita connection in this picture. For a proof of the following Lemma see Lemmas 2.3 and~2.4 of \cite{Fine}. 

\begin{lemma}\label{recognise-LC}
Let $\Sigma \in \Omega^2(M,E^*)$ be wedge-orthogonal. There is a unique metric connection $A$ in $E$ for which $\diff_A \Sigma = 0$. This connection is the pull-back to $E$ of the Levi-Civita connection in $\Lambda^+_{g_\Sigma}$ via the isomorphism $\Sigma \colon E \to \Lambda^+_{g_\Sigma}$. 
\end{lemma}
The equation $\diff_A \Sigma =0$ is a torsion-free condition, see \cite{Fine} for more details. We are now in a position to recognise Einstein metrics in this set-up. The following result is the translation of the familiar fact that a Riemannian 4-manifold is Einstein if and only if the Levi-Civita connection on $\Lambda^+$ is a self-dual instanton. 

\begin{lemma}\label{recognise-Einstein}
Let $\Sigma \in \Omega^2(M,E^*)$ be wedge-orthogonal, with associated metric $g_\Sigma$ and let $A$ be the unique metric connection in $E$ with $\diff_A \Sigma = 0$. Then $g_\Sigma$ is Einstein if and only if $F_A$ is self-dual, i.e.\ $A$ is an instanton with respect to $g_\Sigma$.
\end{lemma}

\subsection{The Plebanski action}

The Plebanski action is a function of three variables, $(A,\Sigma,\Psi)$ where $A$ is a metric connection on the bundle $E$, $\Sigma \in \Omega^2(M,E^*)$  and $\Psi \in \Omega^0(M,S^2_0E)$. Note at this stage these variables are arbitrary; in particular, $\Sigma$ is \emph{not} assumed to be wedge-orthogonal and is unrelated to $A$. 

To describe the action, it is perhaps simplest to choose a local orthonormal frame $e_i$ for $E$. Pick an orientation on $E$ and take the local frame also to be oriented (the freedom in choice of orientation will turn out to be slightly subtle---we will return to this point later). We write $e^i$ for the dual frame of $E^*$ and $\widehat{e}^i$ for the generator of positive rotations about $e_i$, giving a local frame of $\so(E)$. We use the convention that repeated indices are summed over $1,2,3$. 

We write the connection $A$ locally as
\begin{equation}
\diff_A(e_i) = \epsilon_{ijk}A^j \otimes e^k
\label{connection-forms}
\end{equation}
for a triple of 1-forms $A^j$. Similarly, we write $\Sigma = \Sigma_i \otimes e^i$ for a triple of 2-forms $\Sigma_i$ and $\Psi = \Psi^{ij} e_i\otimes e_j$ for a symmetric trace-free matrix valued function $\Psi_{ij}$. We will also need the local expression for the curvarure of $A$. Let
\begin{equation}\label{curvature-forms}
F^i = \diff A^i - \frac{1}{2}\epsilon^i_{\phantom{i}jk}A^j\wedge A^k
\end{equation}
These are the curvature 2-forms, meaning that $(\diff_A)^2(e_i) = \epsilon_{ij}^{\phantom{ij}k}F^j\otimes e_k$. 

\begin{definition}
Fix $\Lambda \in \R$. The \emph{Plebanski action} is the functional
\[
S_P(A,\Sigma,\Psi)
=
\int_M \left\{
F^i \wedge \Sigma_i 
- 
\frac{1}{2} \left(\Psi^{ij} +\frac{\Lambda}{3}\delta^{ij}\right)\Sigma_i \wedge \Sigma_j
\right\}
\]
(Note that the integrand does not depend on the choice of local frame $e_i$ and so makes sense globally.)
\end{definition}

Write $\mathscr{P}$ for the open set of triples $(A,\Sigma,\Psi)$ for which $\tr(\Sigma \wedge \Sigma) >0$. 

\begin{theorem}[Plebanski \cite{Plebanski}]
$(A,\Sigma,\Psi) \in \mathscr{P}$ is a critical point of $S_P$ if and only if $\Sigma$ is wedge-orthogonal and the associated metric $g_\Sigma$ is Einstein. 
\end{theorem}

\begin{proof}
Let $a\in \Omega^1(M,\so(E))$, $\phi \in \Omega^0(M, S^2E_0)$ and $\sigma \in \Omega^2(M,E^*)$. We compute the partial derivatives of $S_P$ in these directions.
\begin{align}
\frac{\del S_P}{\del A}(a) 
	&= \int_M (\diff_A a)^i\wedge \Sigma_i
	\ = \int_M a^i \wedge (\diff_A \Sigma)_i
	\label{dA}
	\\
\frac{\del S_P}{\del \Psi}(\phi) 
	&= - \frac{1}{2} \int_M \phi^{ij} \Sigma_i\wedge \Sigma_j
	\label{dPsi}
	\\
\frac{\del S_P}{\del \Sigma}(\sigma) 
	&= \int_M \left\{
	F^i \wedge \sigma_i 
	- 
	\left(\Psi^{ij}+\frac{\Lambda}{3} \delta^{ij}\right) \Sigma_i \wedge\sigma_j
	\right\}
	\label{dSigma}
\end{align}

From this it follows that
\begin{align}
\frac{\del S_P}{\del A}=0
	&\iff 
		\diff_A \Sigma = 0
		\label{dA=0}\\
\frac{\del S_P}{\del \Psi}=0
	&\iff 
		\Sigma \text{ is wedge-orthogonal}
		\label{dPsi=0}\\
\frac{\del S_P}{\del \Sigma}=0
	&\iff 
		F^i = \left(\Psi^{ij} + \frac{\Lambda}{3}\delta^{ij}\right)\Sigma_j
		\label{dSigma=0}
\end{align}
By \eqref{dPsi=0}, $\Sigma$ determines a metric $g_\Sigma$, with $\Lambda^+ = \im \Sigma$, the image of $\Sigma$. By~\eqref{dA=0} and Lemma~\ref{recognise-LC} the connection $A$ is the pull-back via $\Sigma$  of the Levi-Civita connection. By~\eqref{dSigma=0}, $F_A$ is self-dual. The result follows from Lemma~\ref{recognise-Einstein}.
\end{proof}

When $(A,\Sigma,\Psi)$ is a critical point of $S_P$, equation~\eqref{dSigma=0} tells us that the Einstein metric $g_\Sigma$ has $R_+ = \Psi + \frac{\Lambda}{3}\Id$. In other words, the self-dual Weyl curvature $W^+$ identifies with $\Psi$ via $\Sigma$ whilst the scalar curvature is given by $R = 4 \Lambda$, which gives us the Einstein constant: $\Ric(g_\Sigma) = \Lambda g_\Sigma$. In fact, there is a small subtlety here. Equations~\eqref{dA=0}, \eqref{dPsi=0} and~\eqref{dSigma=0} show that $(A,-\Sigma, -\Psi)$ is a critical point for the Plebanski action with $\Lambda$ replaced by $-\Lambda$. To resolve this ambiguity and determine which sign we should use, note that we fixed an orientation of the bundle $E$. Meanwhile, the bundle $\Lambda^+$ is naturally oriented. We choose the sign of $\Sigma$ in order that the isomorphism $\Sigma \colon E \to \Lambda^+$ is orientation preserving. With this choice of sign, $W^+=\Psi$ and $\Ric(g_\Sigma) = \Lambda g_\Sigma$. 

\subsection{The pure connection action}

At this stage we must restrict to the case $\Lambda \neq 0$. The idea behind the pure connection formulation is that equations~\eqref{dPsi=0} and~\eqref{dSigma=0} can be used to eliminate $\Psi$ and $\Sigma$ and make the action a function purely of the connection~$A$. To do this, we must begin with a connection $A$ which is \emph{definite} (see Definition~\ref{definite}). Write $\D$ for the set of definite connections and $\mathscr{P} = \{(A,\Sigma,\Psi):\tr(\Sigma \wedge \Sigma)>0\}$ for the domain of the Plebanski action. We will define a map $\theta \colon \D \to \mathscr{P}$ of the form $\theta(A) = (A, \Sigma_A, \Psi_A)$ with the following crucial property: \emph{$A$ is a critical point of $S_P \circ \theta$ if and only if $\theta(A)$ is a critical point of $S_P$}. Once $\theta$ has been found, the pure connection action is defined by $S = S_P \circ \theta$.

We first explain how to define $\Sigma_A$ given $A \in \D$. To satisfy equations~\eqref{dPsi=0} and~\eqref{dSigma=0}, we must find a wedge-orthogonal 2-form which is also in the span of the curvature forms. Since $A$ is definite, 
\[
F^i \wedge F^j = Q^{ij} \mu
\]
for some positive-definite matrix $Q^{ij}$ and positive 4-form $\mu$. Let $Y$ be a square-root of $Q$. When $\Lambda>0$ we take $Y$ to be the positive square root, when $\Lambda<0$ we take $Y$ to be the negative square root.  Scaling $\mu$, scales $Q$ and so $Y$; we fix the choice of scale by requiring that $\tr Y = \Lambda$. We denote the corresponding volume form $\mu_A$ to make clear that it is determined by $A$. To ease the notation we write $X = Y^{-1}$. 
Set 
\begin{equation}
\Sigma _i= X_{ij}F^j
\label{Sigma-from-F}
\end{equation}
The $E^*$-valued 2-form $\Sigma_i \otimes e^i$ is globally defined and we denote it by $\Sigma_A \in \Omega^2(M,E^*)$ to emphasise that it is determined by $A$. The point is that the use of $X$ ensures that $\Sigma_A$ is wedge-orthogonal, so \eqref{dPsi=0} is satisfied. Moreover, the $\Sigma_i$ and the $F_i$ span the same subspace, so it will be possible to satisfy equation~\eqref{dSigma=0} by the correct choice of~$\Psi_A$. Indeed given that $\Sigma = XF_A$, we have $F_A = Y \Sigma_A$ and so defining $\Psi_A$ via 
\begin{equation}
\Psi _A+ \frac{\Lambda}{3} \Id = Y
\label{Psi-from-F}
\end{equation}
is the only way to satisfy~\eqref{dSigma=0}.

We now say a word about the sign ambiguity which arose in our discussion of the Plebanski action. The image of $\Sigma_A$ is a definite sub-bundle of $\Lambda^2$, which we denote by $\Lambda^+_A$. The map $\Sigma_A \colon E \to \Lambda^+_A$ is an isomorphism of oriented bundles and it either preserves or reverses orientation. Recall that $\Sigma_A = XF_A$ and the sign of $X$ agrees with that of $\Lambda$. So switching the sign of $\Lambda$ switches the sign of $\Sigma_A$. Given a definite connection only \emph{one} of the two signs of $\Lambda$ is appropriate to consider, that for which the corresponding $\Sigma_A$ is orientation preserving. This sign can be seen directly from the curvature of $A$ as follows.

\begin{definition}
Given a definite connection $A$ and an orientation on $E$, the curvature can be interpreted as an isomorphism $E \to \Lambda^+_A$ of oriented bundles. In our local frame this is the map $e_i \mapsto F^i$. We call $A$ \emph{positive-definite} if this map is orientation preserving and \emph{negative-definite} if this map is orientation reversing. Note that changing the orientation on $E$ changes the sign of this map and so the sign of the definite connection is well-defined. 
\end{definition}

Positive-definite connections are those for which $\Sigma_A$ is orientation preserving when $\Lambda>0$. Negative-definite connections are those for which $\Sigma_A$ is orientation preserving when $\Lambda<0$. From now on, given $\Lambda$ we consider only definite connections with the same sign. When $g$ is Einstein with $R_+$ definite the Levi-Civita connection on $\Lambda^+_g$ is definite and its sign agrees with that of $R_+$. 

We define $\theta \colon \D \to \mathscr{P}$ by $\theta(A) = (A, \Sigma_A, \Psi_A)$. By construction, $(A,\Sigma_A,\Psi_A)$ automatically solves two of the three Euler--Lagrange equations: $\del S_P/\del \Psi = 0 = \del S_P/\del \Sigma$. It follows that $A$ is a critical point of $S_P \circ \theta$ if and only if $\theta(A)$ is a critical point of $S_P$. The remaining Euler--Lagrange equation $\diff_A\Sigma_A = 0$ is now an equation purely for the definite connection $A$. 

Finally, we can write the action $S_P \circ \theta$ purely in terms of the connection. Using the relationship between $\Sigma_A$, $\Psi_A$ and $F_A$ we see that
\[
S_P(\theta(A)) = \frac{1}{2} \int_M X_{ij} F^i \wedge F^j 
	= \frac{\Lambda}{2} \int_M \mu_A
	= \frac{\Lambda}{2} \vol(M,g_A)
\]
where $g_A = g_{\Sigma_A}$ denotes the Riemannian metric determined by $\Sigma_A$ and hence $A$.  
\begin{definition}[See \cite{Krasnov,Fine-Panov-Krasnov}]  
Let $\Lambda \neq 0$ and $A$ be a definite connection of the same sign as $\Lambda$. We call
\[
S(A) = \frac{\Lambda}{2}\vol(M,g_A)
\]
the \emph{pure connection action} of $A$. 
\end{definition}

We summarise this discussion in the following statement.

\begin{theorem}[See \cite{Krasnov,Fine-Panov-Krasnov}]
Fix $\Lambda \neq 0$ and let $A$ be a definite connection of the same sign. $A$ is a critical point of the pure connection $S$ action if and only 
\[
\diff_A \Sigma_A=0
\]
When this happens, $g_A$ is an Einstein metric with $R_+$ definite and of the same sign as $\Lambda$. Moreover all Einstein metrics with $R_+$ definite arise this way.
\end{theorem}

We close this section with some remarks about the existence of definite connections. In \cite{Fine-Panov} it is shown that if a definite connection exists on $M$, then $2\chi(M) + 3\tau(M)>0$. This is ``one half'' of the Hitchin--Thorpe inequality, a necessary condition for the existence of an Einstein metric \cite{Hitchin,Thorpe}. No other obstructions to the existence of definite connections are known. The only known compact examples of \emph{positive} definite connections are on $S^4$ and $\C\P^2$, given by deforming the corresponding Levi-Civita connections of the standard Einstein metrics there. Meanwhile, there are negative definite connections on manifolds on which it is currently unknown if there are Einstein metrics. See \cite{Fine-Panov,Fine-Panov-Krasnov} for examples of definite connections and open questions about them. 

\subsection{The gauge group}

Let $\G \colon E \to E$ denote the group of all maps which send fibres of $E$ to fibres by linear isometries (i.e.\ bundle isomorphisms). We call $\G$ the \emph{gauge group}. An element $\gamma \in \G$ covers a diffeomorphism on the base $M$ and this sets up an exact sequence
\[
1 \to \G_0 \to \G \to \Diff(M) \to 1
\]
where $\G_0$ is the subgroup of gauge transformations which cover the identity on $M$, the ``usual'' gauge group one sees in traditional gauge theory. 

$\G$ acts on the space $\A$ of all connections in $E$, by pull-back. It preserves the open set $\D \subset \A$ of definite connections. The map $A \mapsto g_A$, sending a definite connection to its corresponding metric, is equivariant with respect to the action of $\G$ on $\D$ and the action of $\Diff(M)$ on metrics. It follows that the pure connection action $S \colon \D \to \R$ is $\G$-invariant. 

Given $\eta \in \Lie \G$ we write $L_\eta A$ for the infinitesimal action of $\eta$ at $A$. To compute $L_\eta A$ it is convenient to move to the principal bundle description. $P \to M$ denotes the principal $\SO(3)$-bundle of frames of $E$. $\G$ is then the subgroup of diffeomorphisms of $P$ which commute with the principal $\SO(3)$-action.The Lie algebra $\Lie(\G)$ is a subalgebra of vector fields on $P$. A connection $A$ is an $\SO(3)$-equivariant connection 1-form with values in $\so(3)$. In this picture, the infinitesimal action $L_\eta A$ of $\eta \in \Lie(\G)$ at $A$ is simply the usual Lie derivative. 

We now give a useful formula for $L_\eta A$. Use $A$ to split the vector field $\eta$ into vertical and horizontal parts. When $\eta \in \Lie(\G)$, the vertical part corresponds to a section $\xi \in \Omega^0(M,\so(E))$;  meanwhile, the horizontal part is the $A$-horizontal lift to $P$ of a vector field $v$ on $M$. By Cartan's formula, 
\begin{equation}
L_\eta A
	= 
		\diff (\iota_\eta A) + \iota_\eta (\diff A)
	=
		\diff_A(\xi) + \iota_v F_A
\label{infinitesimal-action}
\end{equation}
(We consider here a right action of $\G$ on $\A$, $A \mapsto \gamma^*A$ whereas usually in gauge theory one considers the corresponding left action. This accounts for the difference in sign of the term $\diff_A(\xi)$ here when compared with standard gauge theory literature.)

\section{The Hessian of the pure connection action}
\label{hessian}

In this section we compute the Hessian of the pure connection action $S$. Recall that $\D$ denotes the set of definite connections (an open set in the affine space of all connections). Given $A \in \D$ which is a critical point of $S$, the Hessian $D^2S$ of $S$ at $A$ is a symmetric bilinear form on $T_A \D = \Omega^1(M, \so(E))$. 

\begin{theorem}\label{hessian-positive}
Let $A$ be a critical point of $S$ corresponding to an Einstein metric for which $R_+<0$. Then $D^2S(a,a) \geq 0$ with equality if and only if $a = L_\eta A$ for some $\eta \in \Lie(\G)$. 
\end{theorem}

It turns out to be easiest to compute the Hessian of $S$ by passing via Plebanski action. Recall that $\mathscr{P} = \{ (A,\Sigma,\Psi) : \tr (\Sigma \wedge \Sigma) >0\}$ denotes the domain of Plebanski action. In the previous section we introduced a map $\theta \colon \D \to \mathscr{P}$ given by $\theta(A)= (A,\Sigma_A, \Psi_A)$ where $\Sigma_A$ and $\Psi_A$ are defined by~\eqref{Sigma-from-F} and~\eqref{Psi-from-F}. The pure connection action is the pull-back of Plebanski action by $\theta$. It follows that at a critical point $A$, the Hessian of the pure connection action is the pull-back by $\theta$ of the Hessian of $S_P$ at $\theta(A)$. To compute $D^2S$ we will begin with the Hessian of Plebanski action at an arbitrary critical point $(A, \Sigma, \Psi)$ (not necessarily of the form $\theta(A)$ for a definite connection $A$). Then we pull this back to $\D$ via $\theta$. To complete the proof of Theorem~\ref{hessian-positive}, we will also bring gauge fixing into the picture. 

\subsection{The Hessian before gauge fixing}
\label{before-gauge}

We begin with the Hessian of the Plebanski action. Fix a critical point $(A,\Sigma,\Psi)$ of $S_P$ and choose a tangent direction $t = (a,\sigma,\phi)$. Write $D^2S_P(t,t)$ for the Hessian of $S_P$ in this direction. Differentiating the equations~\eqref{dA}, \eqref{dPsi} and~\eqref{dSigma} again with respect to $(a,\sigma,\phi)$ (and recalling our conventions \eqref{connection-forms} and \eqref{curvature-forms}) we see that $D^2S_P(t,t)$ is equal to 
\begin{equation}\label{Hessian-Plebanski}
  \int_M \left\{
- \epsilon_{ij}^{\phantom{ij}k}a^i\wedge a^j \wedge \Sigma_k
+
2 \left( (\diff _A a)^i - \phi^{ij} \Sigma_j\right) \wedge \sigma_i
-
\left(\Psi^{ij} + \frac{\Lambda}{3} \delta^{ij} \right) \sigma_i \wedge \sigma_j
\right\}
\end{equation}

We now consider the Hessian of the pure connection action $S$. Let $A$ be a definite connection which is a critical point of $S$ so that $\theta(A) = (A, \Sigma_A, \Psi_A)$ is a critical point of $S_P$. Given an infinitesimal change $a$ of $A$, the Hessians of $S$ and $S_P$ are related by
\[
D^2S(a,a) = D^2S_P(\theta_*(a), \theta_*(a))
\]
Write $\theta_*(a) = (a,\sigma,\phi)$ where $\sigma$ and $\phi$ denote the corresponding infinitesimal changes in $\Sigma_A$ and $\Psi_A$ respectively. To proceed we will need the precise dependence of $\sigma$ and $\phi$ on $a$. To obtain this note that the tangent vector $(a,\sigma, \phi)$  preserves the conditions~\eqref{dPsi=0} and~\eqref{dSigma=0}. Differentiating \eqref{dPsi=0} we see that
\begin{equation}
\sigma_i \wedge \Sigma_j + \Sigma_i \wedge \sigma_j
-
\frac{2}{3} (\sigma_k \wedge \Sigma_k) \delta_{ij}
=0
\label{double-sigma}
\end{equation}
Differentiating \eqref{dSigma=0} gives
\begin{equation}\label{pre-gives-sigma}
(\diff_A a)^i - \phi^{ij}\Sigma_j - Y^{ij}\sigma_j
=
0
\end{equation}
(Recall that $Y = \Psi_A + \frac{\Lambda}{3} \Id$.) Rearranging this gives
\begin{equation}\label{give-sigma}
\sigma_i = X_{ij}\left(
(\diff_A a)^j - \phi^{jk}\Sigma_k
\right)
\end{equation}
(Recall that $X = Y^{-1}$.) From here we can solve for $\phi$ in terms of $\diff_Aa$ via \eqref{double-sigma}. We do this as follows.

Using the local frame $\Sigma^j$ for $\Lambda^+$ we write the self-dual part of $\sigma_i$ as $\sigma^+_i = Z_{i}^{\phantom{i}j}\Sigma_j$ for some $3\times 3$-matrix valued function $Z$. Similarly, we put $(\diff_A^+a)^i = N^{ij}\Sigma_j$, so~\eqref{give-sigma} reads
\begin{equation}\label{Z}
Z_{i}^{\phantom{i}j} = X_{ik}(N^{kj} - \phi^{kj})
\end{equation}
Now an arbitrary $3\times 3$-matrix has three irreducible parts: the symmetric trace-free component, the skew component and the trace component. Equation~\eqref{double-sigma} says that the symmetric trace-free part of $Z$ is zero. This uniquely determines $\phi^{ij}$ in terms of $N^{ij}$ and $X_{ij}$. To see this write $L \colon S^2_0E \to S^2_0E$ for the linear bundle map $L(\phi) = S^2_0(X\phi)$ where $S^2_0$ denotes the projection operator sending a matrix to the trace-free symmetric part. Then $\phi$ must solve $L(\phi) = S^2_0(XN)$. This has a unique solution $\phi \in S^2_0E$ provided $L$ is injective. Suppose then that $L(\psi) = 0$ for some $\psi \in S^2_0E$ or, in other words, that 
\[
\psi + X\psi X^{-1} = f X^{-1}
\]
for some function $f$. Taking the trace and using the fact that $\tr X^{-1} = \Lambda$ we see that $f=0$. On the one hand the eigenvalues of the similar matrices $\psi$ and $X\psi X^{-1}$ are equal, on the other hand the fact that $\psi = -X\psi X^{-1}$ shows the eigenvalues are opposite. The only conclusion is that $\psi =0$ and so $L$ is an isomorphism as required. 

We can now give the formula for the Hessian of the pure connection action.

\begin{lemma}\label{Hessian-before-gauge-fix}
Let $A \in \D$ be a critical point of $S$ and let $a \in T_A \D$. The Hessian of $S$ in the direction $a$ is given by
\begin{equation*}
D^2S(a,a) 
  =
  \int_M \left[ 
  -\epsilon_{ij}^{\phantom{ij}k}a^i \wedge a^j \wedge \Sigma_k
  +
  X_{ij} 
  \left((\diff_A a)^i - \phi^{ik}\Sigma_k\right)
  \wedge
  \left((\diff_A a)^j - \phi^{jk}\Sigma_k\right)
\right]
\end{equation*}
where $\phi\in S^2_0E$ is the unique solution to $S^2_0(X\phi) = S^2_0(XN)$ where $d_A^+a = N\Sigma$. 
\end{lemma}

\begin{proof}
We have seen that $\theta_*(a) = (a,\sigma,\phi)$ where $\sigma$ and $\phi$ satisfy \eqref{give-sigma}. We now substitute \eqref{give-sigma} into \eqref{Hessian-Plebanski} and use the fact that 
\[
X_{ij}\left(\Psi^{ij} + \frac{\Lambda}{3} \delta^{ij}\right) = X_{ij}Y^{ij} = \delta_{i}^j
\]
This gives the claimed expression for $D^2S(a,a)$. The condition on $\phi$ is the content of the discussion after~\eqref{Z}   
\end{proof}

\subsection{The Hessian after gauge fixing}\label{after-gauge}

In this section we will explain how the formula of Lemma~\ref{Hessian-before-gauge-fix} for $D^2S(a,a)$ simplifies when we impose additional gauge fixing conditions on $a$. Before discussing this however we need a short digression. Given a wedge-orthogonal $\Sigma$, we can define a triple of almost complex structures $J_i$ by ``raising an index'' on $\Sigma_i$ with $g_\Sigma$. (It is for this reason that that the metric $g_\Sigma$ is scaled so that $|\Sigma_i|^2=2$.) Equivalently, one can define the action of $J_i$ on 1-forms via the equation
\begin{equation}\label{almost-complex}
	J_i(\alpha) = * (\Sigma_i \wedge \alpha)
\end{equation}
One can check that the $J_i$ satisfy the quaternion relations:
\begin{equation}\label{quaternion}
J_i J_j = \epsilon_{ij}^{\phantom{ij}k} J_k - \delta_{ij}
\end{equation}
At this point we use crucially that $\Sigma \colon E \to \Lambda^+$ is orientation-preserving; if it were orientation-reversing then $-J_i$ would satisfy~\eqref{quaternion}.

We will use the freedom to act by gauge transformations to ensure that the infinitesimal change in connection $a^i$ satisfies the following two conditions:

\begin{definition}
Let $A$ be a definite connection which is a critical point of the pure connection action $S$. Let $\Sigma_A$ be determined by $A$ via~\eqref{Sigma-from-F}. We say that $a \in \Omega^1(M,\so(E))$ is in \emph{horizontal gauge} if 
\begin{equation}
\Sigma_i \wedge a^i  = 0
\label{horizontal-gauge}
\end{equation}
In terms of the almost complex structures, this is equivalent to $J_ia^i =0$. Given $a$ in horizontal gauge, we say that it is also in  \emph{vertical gauge} if in addition:
\begin{equation}
\epsilon^{ij}_{\phantom{ij}k}\Sigma_j \wedge (\diff_A a)^k = 0
\label{coulomb-gauge}
\end{equation}
When $a$ is in both horizontal and vertical gauge we say that \emph{$a$ is fully gauge fixed}.
\end{definition}
We briefly explain the names. Loosely speaking, gauge transformations are made up of diffeomorphisms on $M$ and purely ``vertical'' gauge transformations which preserve each fibre of $E$. The horizontal gauge fixing condition gives a complement to the infinitesimal action of diffeomorphisms, i.e.\ of horizontal lifts of vector fields from $M$. Now for those $a$ which are in horizontal gauge, the second condition~\eqref{coulomb-gauge} is equivalent to Coulomb gauge, $\diff_A^*a = 0$, as the following lemma shows. This then gives a complement to the infinitesimal action of vertical gauge transformations. 

\begin{lemma}
If $a$ is in horiztonal gauge and $\diff_A \Sigma_A=0$ then
\begin{align}
* a^i  &= \epsilon^{ij}_{\phantom{ij}k}\Sigma_j \wedge a^k\label{hodge-spin-gauge}\\
(\diff_A^* a)^i &= -*\left( \epsilon^{ij}_{\phantom{ij}k} \Sigma_j\wedge(\diff_A a)^k\right)\label{d-star-spin-gauge}
\end{align}
\end{lemma}
\begin{proof}
In terms of the almost complex structures $J_i$, the horizontal gauge condition \eqref{horizontal-gauge} reads $J_ia^i =0$ or, equivalently, $a^i = -\epsilon^{ijk}J_ja_k$. Taking the Hodge star, and using the fact that $*^2=-1$ on 3-forms, we obtain $*a^i = \epsilon^{ij}_{\phantom{ij}k}\Sigma_j \wedge a^k$. Now, since $\diff_A \Sigma_A =0$ and $\diff_A^* = -*\diff_A *$ equation \eqref{d-star-spin-gauge} follows from \eqref{hodge-spin-gauge}.
\end{proof}

Before discussing how to enforce these two conditions, we show the effect they have on the Hessian of the pure connection action. 

\begin{proposition}[The gauge-fixed Hessian]
\label{Hessian-positive-in-gauge}
Let $A \in \D$ be a critical point of the pure connection action $S$. Let $a$ be an infinitesimal change in $A$ which satisfies both gauge fixing conditions~\eqref{horizontal-gauge} and~\eqref{coulomb-gauge}. Then the Hessian of the pure connection action in this direction is given by
\begin{equation}
D^2S(a,a)
= 
\int_M \left(
|a|^2 - X_{ij} \left\langle (\diff_A^-a)^i, (\diff_A^-a)^j \right\rangle\right) 
\mu_A
\end{equation}
In particular, when  $R_+$ is negative definite, so that $X$ is negative definite, 
\[
D^2S(a,a)\geq 0
\] 
with equality if and only if $a=0$. 
\end{proposition}
\begin{proof}
Recall the formula for $D^2S_A(a,a)$ given in Lemma~\ref{Hessian-before-gauge-fix} above. When $a$ is in horizontal gauge, the first term in the integrand is
\begin{equation}
- \epsilon_{ij}^{\phantom{ij}k}a^i\wedge a^j \wedge \Sigma_k
	=
		a^i \wedge * a^i
	=
		|a|^2\, \mu_A	
\end{equation}

We now turn to the second term in the integrand, namely
\[
X_{ij}
\left((\diff_A a)^i - \phi^{ik}\Sigma_k\right)
\wedge
\left((\diff_A a)^j - \phi^{jk}\Sigma_k\right)
\]
Here $\phi \in S^2_0E$ is determined by $a$ via $S^2_0(X\phi) = S^2_0XN$ where $(\diff_A^+a)^i = N^{ij}\Sigma_j$. The Coulomb gauge condition~\eqref{coulomb-gauge} says that the skew-symmetric part of $N$ vanishes. Meanwhile, differentiating the horizontal gauge condition~\eqref{horizontal-gauge} and using $\diff_A \Sigma_A=0$ gives
\begin{equation}
\Sigma_i \wedge (\diff_A a)^i = 0
\end{equation}
which means that the trace of $N$ vanishes. It follows that $N$ is itself a symmetric trace-free matrix and so $\phi = N$. It follows that
\begin{equation}\label{sigma-asd}
(\diff_A a)^i - \phi^{ij}\Sigma^j = (\diff^-_A a)^i
\end{equation}
is purely anti-self-dual. The second term is thus
\[
X_{ij} (\diff_A^- a)^i\wedge (\diff_A^-a)^j
\]
For anti-self-dual forms, the wedge product is minus the pointwise inner product which completes the proof.
\end{proof}

\subsection{Gauge fixing}\label{gauge-fixing}

We now show that the space of fully gauge fixed $a$ defines a complement to the infinitesimal gauge action. Recall that elements $\eta \in \Lie(\G)$ correspond to pairs $(v,\xi)$ where $\xi \in \Omega^0(M,\so(E))$ and $v$ is a vector field on $M$. The infinitesimal action of $\eta$ at $A$ is $L_\eta(A) = \iota_v F_A + \diff_A \xi$

Let 
\begin{equation}
W 
  = 
    \left\{ 
    a \in C^\infty(\Lambda^1\otimes E) : 
    a \textrm{ is fully gauge fixed, \eqref{horizontal-gauge} and \eqref{coulomb-gauge} hold}
    \right\}
\end{equation}
The goal of this section is to prove the following:
\begin{proposition}\label{linear-gauge-fix}
There is a direct sum decomposition
\begin{equation}
\Omega^1(M,\so(E)) = \{L_\eta A : \eta \in \Lie(\G)\} \oplus W
\end{equation}
\end{proposition}

Before giving the proof we first try and lighten the notation a little. Write 
\[
p \colon \Lambda^1 \otimes \so(E) \to \Lambda^1,
\quad
p(a) = J_ia^i
\] 
Recall that the horizontal gauge condition \eqref{horizontal-gauge} is equivalent to $p(a)=0$. Write 
\[
q \colon TM \to \Lambda^1\otimes \so(E),
\quad
q(v) = \iota_v F_A
\]
To prove Proposition~\ref{linear-gauge-fix} we must show that for any $a$ there is a unique solution $(v,\xi)$ to the pair of equations
\begin{align}
p\left( a + q(v) + \diff_A \xi\right) &=0\label{alg-gauge-fix}\\
\diff_A^*\left( a + q(v) + \diff_A \xi\right) &=0 \label{diff-gauge-fix}
\end{align}
We will show that this leads to an elliptic equation for $\xi$. As we will see shortly, the composition $p\circ q$ is invertible. Assuming this for the moment, for any choice of $a$ and $\xi$, there is a unique choice of $v$ solving~\eqref{alg-gauge-fix}, namely
\begin{equation}
\label{v-in-terms-of-xi}
v = - (p\circ q)^{-1}\circ p\left((\diff_A \xi) + a\right)
\end{equation}
Write 
\[
\Pi = 1 - q \circ (p \circ q)^{-1} \circ p
\]
It is simple to check that $\Pi$ is a projection operator, projecting $\Lambda^1\otimes E$ onto $\ker(p)$ against $\im(q)$. In other words, $\Pi(a)$ is what one obtains by putting $a$ in horizontal gauge using the infinitesimal action of tangent vectors. Now substituting~\eqref{v-in-terms-of-xi} into~\eqref{diff-gauge-fix} we see that~\eqref{alg-gauge-fix} and~\eqref{diff-gauge-fix} are equivalent to the following single equation for $\xi$:
\begin{equation}\label{gauge-fix-PDE}
\diff_A^*\Pi \diff_A \xi
  =
  -  \diff_A^* \Pi a
\end{equation}
We will prove that~\eqref{gauge-fix-PDE} has a unique solution in three steps: by showing that $\diff_A^*\Pi\diff_A$ is elliptic; that it has index zero and finally that it is injective. 

The symbol of $\diff^*_A \Pi \diff_A$ in the direction of the covector $\alpha$ is the map
\[
	\Sym(\alpha) \colon E \to E,\quad \Sym(\alpha)(\xi) = \left\langle \alpha, \Pi(\alpha \otimes \xi)\right\rangle
\]
To compute this we first give a formula for $\Pi$. Along the way we show that $p\circ q$ is invertible. Recall that $F^i = Y^{ij}\Sigma_j$ where $Y = \Psi+\frac{\Lambda}{3}\Id$. Our curvature assumption that $R_+<0$ is equivalent to $Y<0$. 

\begin{lemma}
We have $p\circ q(v) = -\Lambda v^\flat$  where $v^\flat$ denotes the covector which is metric dual to the tangent vector $v$. It follows that $(p\circ q)^{-1}(\alpha) = -\Lambda^{-1} \alpha^\#$ where $\alpha^\#$ is the tangent vector metric dual to the covector $\alpha$.
\end{lemma}
\begin{proof}
This is a direct calculation, using the quaternion relations~\eqref{quaternion}. 
\begin{align*}
p\circ q(v)
  &=
    J_i( \iota_v F^i)\\
  &=
    J_i(Y^{ij}\iota_v\Sigma_j)\\
  &=
    Y^{ij} J_iJ_j(v^\flat)\\
  &=
    Y^{ij}\left(\epsilon_{ij}^{\phantom{ij}k}J_k(v^\flat) - \delta_{ij} v^\flat\right)\\
  &=
    - \tr(Y)v^\flat\\
  &=
    - \Lambda v^\flat \qedhere
\end{align*} 
\end{proof}

\begin{lemma}
The projection $\Pi$ has the following expression in terms of the self-dual curvature matrix~$Y^{ij}$.
\[
(\Pi a)^i = a^i + \Lambda^{-1}\epsilon_{kj}^{\phantom{kj}l}Y^{ik}J_l(a^j) 
        - \Lambda^{-1}Y^{ij}a^j
 \]	
\end{lemma}
\begin{proof}
Again, this is direct calculation.	
\begin{align*}
(\Pi a)^i
  &=
    a^i - q\circ (p\circ q)^{-1} \circ p(a)^i\\
  &=
    a^i + \Lambda^{-1} \iota_{(J_ja^j)^\#}F^i\\
  &=
    a^i + \Lambda^{-1}Y^{ik} \iota_{(J_ja^j)^\#}\Sigma^k\\
  &=
    a^i + \Lambda^{-1}Y^{ik} J_kJ_j(a^j)\\
  &=
    a^i + \Lambda^{-1}\epsilon_{kj}^{\phantom{kj}l}Y^{ik}J^l(a^j) 
        - \Lambda^{-1}Y^{ij}a^j
        \qedhere
\end{align*}
\end{proof}

\begin{lemma}
When $R_+$ is negative definite, $\diff_A^*\Pi \diff_A$ is an elliptic operator on sections of $\so(E)$.
\end{lemma}
\begin{proof}
We use the above formula for $\Pi$ to compute the symbol acting on the covector $\alpha$. Write $\xi = \xi^i \otimes \widehat{e}_i$. 
\begin{align*}
\Sym(a)(\xi)
  &=
    \left\langle \alpha, \Pi(\alpha \otimes \xi) \right\rangle\\
  &=
    \xi^i|\alpha|^2
    + 
    \Lambda^{-1}\epsilon_{kj}^{\phantom{kj}l}\xi^jY^{ik}\left\langle \alpha, J_l(\alpha) \right\rangle
    -
    \Lambda^{-1}Y^{ij}\xi^j|\alpha|^2\\
  &=
    \left(\xi^i - \Lambda^{-1}Y^{ij}\xi^j\right)|\alpha|^2
\end{align*}
where in the last line we have used that $\alpha$ and $J_l(\alpha)$ are orthogonal. It follows that, for $\alpha \neq 0$, the kernel of $\Sym(\alpha)$ is exactly the space of eigenvectors of $Y$ with eigenvalue $\Lambda$. But $\Lambda = \tr(Y)$ and all eigenvalues of $Y$ are \emph{strictly} negative, so none can be equal to~$\Lambda$. It follows that $\Sym(\alpha)$ is invertible and so $\diff_A^*\Pi\diff_A$ is elliptic.
\end{proof}

\begin{lemma}
When $R_+$ is negative definite, the index of $\diff_A^*\Pi \diff_A$ is zero.
\end{lemma}

\begin{proof}
If $\Pi$ were self adjoint then $\diff_A^*\Pi\diff_A$ would be also and it would follow immediately that the index vanishes. This is not the case however because, in general, $\ker p$ and $\im q$ are not orthogonal subspaces. To remedy this, consider the path $Y_t^{ij} = t \Psi^{ij} + \frac{\Lambda}{3} \delta^{ij}$ of endomorphisms and the corresponding path $F_t^i = Y_t^{ij}\Sigma^j$ of sections of $\Lambda^2\otimes \so(E)$. The forms $F_t$ define a path of maps $q_t \colon TM \to \Lambda^1 \otimes \so(E)$ given by $q_t(v) = \iota_vF_t$. Since $Y=Y_1$ is negative-definite, the same is true for $Y_t$ for all $t \in [0,1]$. Moreover, $\tr Y_t = \Lambda$. These were the only properties we used in the above arguments. It follows that $\Pi = 1 - q_t \circ (p\circ q_t)^{-1} \circ p$ is a projection of $\Lambda^1\otimes \so(E)$ onto $\ker p$ against $\im q_t$ and, moreover, that $\diff_A^*\Pi_t \diff_A$ is a path of elliptic operators. It follows that our operator has the same index as $\diff_A^*\Pi_0\diff_A$. But when $t=0$, $\im q_0$ \emph{is} the orthogonal complement of $\ker p$ and so $\Pi_0$ is a self-adjoint projection. It follows that the index of $\diff_A^*\Pi_0 \diff_A$, and hence of all the $\diff_A^*\Pi_t \diff_A$, vanishes for $t \in [0,1]$.
\end{proof}

\begin{proposition}
When $R_+$ is negative-definite, the map 
\[
\diff_A^*\Pi\diff_A \colon C^\infty(E) \to C^\infty(E)
\] 
is an isomorphism.	
\end{proposition}
\begin{proof}
Now that we know $\diff_A^*\Pi\diff_A$ is elliptic of index zero, if suffices to show that it is injective. Suppose then that $\diff_A^*\Pi \diff _A \xi =0$	. This means that $\Pi \diff_A \xi$ is fully gauge fixed. By Proposition~\ref{Hessian-positive-in-gauge} it follows that either $\Pi \diff_A \xi=0$, or the Hessian in this direction is non-zero. Now, by definition of $\Pi$, there exists $v \in C^\infty(TM)$ such that $\Pi\diff_A \xi = \diff_A \xi + \iota_v F$. In other words, $\Pi\diff_A \xi$ is pure gauge. Since the pure connection action is gauge invariant, the Hessian vanishes in pure gauge directions. It follows that $\Pi \diff_A \xi =0$ or, in other words, that $\iota_v F =- \diff_A \xi$. 

Now the infinitesmial gauge transformation $\xi$ doesn't change the induced metric $g_A$ and so $v$ is a Killing field. But the metric has negative Ricci curvature and so $v$ must vanish\footnote{We mention the simple proof of this standard fact for completeness. If $v$ is Killing, Bochner's formula gives $\frac{1}{2}\Delta|v|^2 = |\nabla v|^2 - \Ric(v,v)$. When $\Ric<0$ and the manifold is compact, integrating this equality shows $v$ must vanish.}. It follows that $\diff_A \xi =0$. From this we see that $F_A(\xi)=0$, or in other words that $Y^{ij}\xi^j=0$. Since $Y$ is invertible we obtain $\xi=0$ as required. 
\end{proof}

This Proposition, together with the discussion immediately after equations~\eqref{alg-gauge-fix} and~\eqref{diff-gauge-fix}, completes the proof of Proposition~\ref{linear-gauge-fix}, which shows that, on the linear level, one can always fully gauge fix $a$ via the infinitesimal action of the gauge group. 

\subsection{The proof of Theorem~\ref{hessian-positive}}

We now give the proof of Theorem~\ref{hessian-positive}. Let $A$ be a critical point of the pure connection action $S$ and assume that the corresponding Einstein metric has $R_+<0$. Let $a$ be an infinitesimal deformation of $A$. By Proposition~\ref{linear-gauge-fix} we can write $a = b+L_{\eta}A$ for $b$ fully gauge-fixed and $\eta \in \Lie(\G)$. Since $S$ is $\G$ invariant, $D^2S(L_{\eta}A, \cdot )=0$. It follows that $D^2S(a,a) = D^2S(b,b)$. Theorem~\ref{hessian-positive} now follows from Proposition~\ref{Hessian-positive-in-gauge}.

\section{The proof of local rigidity}
\label{infinitesimal-rigidity}

In this section we show how Theorem~\ref{hessian-positive} implies our main result Theorem~\ref{main-result}. We give the statement again for convenience.

\begin{theorem*}
Let $g$ be an Einstein metric on a compact orientable 4-manifold $M^4$, with either $R_{+}<0$ or $R_{-}<0$. Then $g$ is locally rigid.
\end{theorem*}

\begin{proof}
Let $h$ be an infinitesimal Einstein deformation of $g$. We must find a vector field $v$ on $M$ for which $L_vg =h$. The idea of the proof is to show that the metric deformation $h$ is induced by a deformation $a$ of the definite connection $A$ corresponding to $g$ and, moreover, with $a \in \ker D^2S$. This will imply that $a$ is pure gauge and so the same will be true of $h$.

It is certainly not true that an \emph{arbitrary} metric deformation can be achieved by deforming the connection. This can be seen by counting degrees of freedom. The space of connections is modelled on $\Omega^1(M, \so(E))$ and so has functional dimension 12. The gauge group $\G$ meanwhile has functional dimension 7, so one can only expect to reach a space of metrics with effective functional dimension 5. Meanwhile metrics have functional dimension 10 of which diffeomorphisms account for 4 dimensions, leaving effective functional dimension 6 which is larger. This means we must use in an essential way the fact that $h$ is an infinitesimal Einstein deformation.

Write $\M$ for the space of Riemannian metrics and $G \colon \D \to \M$ for the map which sends a definite connection to the corresponding Riemannian metric. This map is equivariant for the action of $\G$ on $\D$ and $\Diff(M)$ on $\M$. Let $A \in \D$ be the critical point of the pure connection action $S$ with $G(A)=g$. We seek $a \in T_A\D$ such that $D_A G(a) = h$ and with $a\in \ker D^2S$. From here, Theorem~\ref{hessian-positive} shows that $a = L_\eta A$ for some $\eta \in \Lie(\G)$. Writing $v \in C^\infty(M,TM)$ for the pushforward of $\eta$ to $M$, equivariance of $G$ then implies that $h=L_vg$.

To proceed let $\widehat{\Sigma}(t) \colon E \to \Lambda^2$ be a path of wedge-orthogonal 2-forms with $\widehat\Sigma(0)$ corresponding to $A$ via~\eqref{Sigma-from-F}. Write the corresponding path of metrics as 
\[
\widehat{g}(t) = g_{\widehat{\Sigma}(t)}
\] 
and choose $\widehat{\Sigma}(t)$ so that $\widehat{g}'(0) =h$. 

Now let $A(t)$ be the pull-back to $E$ of the $\widehat{g}(t)$-Levi-Civita connection on $\Lambda^+(\widehat{g}(t))$ via $\widehat{\Sigma}(t)$. By construction $A(0) =A$. Each connection $A(t)$ is again definite (at least for small $t$). Write $a = A'(0) \in T_A\D$. $A(t)$ determines a wedge-orthogonal map $\Sigma(t) = \Sigma_{A(t)}$ via~\ref{Sigma-from-F} and a path of metrics 
$g(t) = g_{A(t)}$. 

It is important to note that $g(t)$ need not be equal to the path $\widehat{g}(t)$ we first started with. We will prove that when $h$ is an infinitesimal Einstein deformation, that these paths agree to first order: $g'(0) = \widehat{g}'(0)$. Write $\sigma = \Sigma'(0)$ and $\widehat{\sigma} = \widehat{\Sigma}'(0)$. In fact, we will show that $\sigma = \widehat{\sigma}$, from which it follows that $g'(0) = \widehat{g}'(0)$. 

The 2-form $\widehat{\Sigma}(t)$ is wedge-orthogonal. Differentiating this condition with respect to $t$ and evaluating at $t=0$ we get 
\begin{equation}
\widehat{\sigma}_i \wedge \Sigma_j + \Sigma_j \wedge \widehat{\sigma}_i -\frac{2}{3} (\widehat{\sigma}_k \wedge \Sigma_k) \delta_{ij}
= 0
\label{double-hat-sigma}
\end{equation}
(Cf.\ equation \eqref{double-sigma}.) Next write
\[
F_{A(t)} = B(t) + C(t)
\]
where $B(t)$ is $\widehat{g}(t)$-self-dual and $C(t)$ is $\widehat{g}(t)$-anti-self-dual. Since $\widehat{g}(0)$ is Einstein and $h$ is an infinitesimal Einstein deformation, $C(t) = O(t^2)$. We write (locally) $B^i = \hat{Y}^{ij}(t)\widehat{\Sigma}_j(t)$. At $t=0$, $\widehat{Y}(0) = Y$, the matrix representative for $R_+$ of the original Einstein metric $\widehat{g}(0)=g$. Symmetries of the Riemann tensor ensure that $\widehat{Y}(t)$ is symmetric for all $t$. Moreover, since $h$ is an infinitesimal Einstein deformation, $\tr \widehat{Y} = \Lambda + O(t^2)$. If follows that $\widehat{Y}'(0)$ is trace-free.  Differentiating  $F_{A(t)}$ with respect to $t$ and evaluating at $t=0$ we get
\begin{equation}
(\diff_A a)^i 
	= 
		(\widehat{Y}')^{ij}(0)\Sigma_i 
		+ 
		Y^{ij}\widehat{\sigma}_j
\label{pre-gives-hat-sigma}
\end{equation}

Meanwhile, the path $A(t)$ of definite connections determines a path $(\Sigma(t),\Psi(t))$ via~\eqref{Sigma-from-F} and~\eqref{Psi-from-F}. The tangent $(a,\sigma,\phi)$ to the whole path $(A,\Sigma_A,\Psi_A)(t)$ of Plebanski variables at $t=0$ satisfies the equations~\eqref{double-sigma} and~\eqref{pre-gives-sigma}. Moreover, given $a$, we saw that there are unique $\sigma$ and $\phi$ solving~\eqref{double-sigma} and~\eqref{pre-gives-sigma}. But these equations are identical to~\eqref{double-hat-sigma} and~\eqref{pre-gives-hat-sigma} satisfied by $\widehat{\sigma}$, with $\phi$ replaced by $\widehat{Y}'(0)$. It follows that $\sigma = \hat{\sigma}$ as claimed (and $\phi = \widehat{Y}'(0)$ but we will not use this).

Finally we have to check that the tangent vector $a \in T_A \D$ is in $\ker D^2S$. This is equivalent to $a$ being tangent to the critical locus of $S$ which is in turn equivalent to $\diff_{A(t)} \Sigma(t) =O(t^2)$. Recall that $A(t)$ is the pull-back of the Levi-Civita connection via $\widehat{\Sigma}(t)$. It follows from Lemma~\ref{recognise-LC} that $\diff_{A(t)} \widehat{\Sigma}(t) = 0$ for all $t$. But we have just seen that $\widehat{\Sigma}(t) =\Sigma(t) + O(t^2)$ and so $\diff_{A(t)} \Sigma(t) =O(t^2)$ and $a \in \ker D^2S$ as claimed.  
\end{proof}

\end{document}